\documentclass{cimart}

\usepackage[noabbrev]{cleveref}
\usepackage{mathtools}
\usepackage{derivative}
\usepackage{mleftright}
\usepackage[only,llbracket,rrbracket]{stmaryrd}

\newcommand{\C}{\mathbb C}
\newcommand{\coeff}{\genfrac{ \lbrack }{ \rbrack }{0pt}{}}
\newcommand{\inv}{\frac{1}}
\newcommand{\N}{\mathbb N}
\renewcommand{\phi}{\varphi}
\newcommand{\qbinom}[2]{\binom{#1}{#2}_{\mkern-7mu q}}
\newcommand{\Stir}{\genfrac{ \{ }{ \} }{0pt}{}}
\newcommand{\x}{\text{\ttfamily\mdseries\upshape x}}
\newcommand{\Z}{\mathbb Z}

\DeclareMathOperator{\itlog}{itlog}

\NewDocumentCommand\DeclareDelimiter{mmm}{
  \NewDocumentCommand{#1}{m}{%
    \mleft\ifblank{#2}{.}{#2}##1\mright#3
  }
}

\DeclareDelimiter{\pa}{\lparen}{\rparen}
\DeclareDelimiter{\cbra}{\{}{\}}
\DeclareDelimiter{\rbar}{}{\rvert}
\DeclareDelimiter{\bbra}{\llbracket}{\rrbracket}

\crefname{theorem}{theorem}{theorems}
\crefname{lemma}{lemma}{lemmas}

\SetSymbolFont{stmry}{bold}{U}{stmry}{m}{n}

\title{Explicit expressions for iterates of power series}

\authors{Kei Beauduin}

\authorinfo{Graduate School of Mathematics, Kyushu University, Fukuoka, Japan}{beauduin.kei.953@s.kyushu-u.ac.jp}

\abstract{%
    We present several formulas for both the discrete and fractional iterates of an invertible power series $f,$ using a new unifying approach based on umbral calculus. Known formulas are extended, and their proofs simplified, while new expressions are introduced. In particular, by employing $q$-calculus identities, we eliminate the requirement for $f'(0)$ to equal $1$ and the resulting general expressions for the iterative logarithm are also obtained.
    }



\VOLUME{35}
\YEAR{2027}
\ISSUE{1}
\NUMBER{1}
\DOI{https://doi.org/10.46298/cm.17047}
\licence{CC BY-SA 4.0}
\editinfo{December 08, 2025}{September 22, 2026}{Rigo Michel}
\acknowledgments{%
    We would like to thank A.~Eremenko for pointing us toward a historical reference. Large language models were used to assist in the critical review of drafts of this paper. Their use did not contribute mathematical ideas. All suggestions were evaluated by the author, who assumes full responsibility for the final manuscript.
}

\begin{document}

\section*{Introduction}
\addcontentsline{toc}{section}{Introduction}

The ring of formal power series is closed under addition, multiplication and, as a corollary, under composition with a series with zero constant term. This makes it meaningful to study repeated self-composition -- \emph{iteration} -- of a compositionally invertible power series $f,$ denoted $f^s$ for $s\in\N := \{0, 1, 2, \dots\}.$ Investigations of this problem go back to Condorcet (1777) \cite{condorcet1778}. This topic became even more intriguing once authors noticed the natural extension of iteration to complex values, sometimes known as \emph{fractional iteration}. Despite its long history, explicit formulas for fractional iterates remain relatively rare in the literature, owing to the substantial difficulty of tackling the issue directly. To overcome these challenges, Schröder (1870) \cite{schroeder1870} and Jabotinsky (1963) \cite{jabotinsky1963} used different techniques yielding distinct formulas for the fractional iterates of $f$ when $f'(0) = 1.$ Schröder relied on an intriguing identity of Cayley \cite{cayley1860}, whereas Jabotinsky used a matrix representation of power series due to Bennett \cite{bennett1915,jabotinsky1947}, commonly referred to as \emph{Jabotinsky matrices}. Unbeknownst to Jabotinsky, Tambs Lyche \cite{tambslyche1927} had arrived at the same results in 1927, with a more involved proof, but under fewer restrictions on $f'(0).$

In \Cref{s:preli}, we introduce a new technique from the theory of \emph{umbral calculus} that unifies the two approaches. Specifically, in \Cref{s:frac1} we employ a generalization of Jabotinsky matrices that enables the derivation of both Jabotinsky's and Schröder's formulas. Then, utilizing \emph{$q$-calculus}, in \Cref{s:schrodergen,s:jabogen} we extend these findings to situations where $q := f'(0)$ is not a root of unity. This not only shortens Tambs Lyche's original proof but also establishes the general form of Schröder's formula. In addition, in \Cref{s:monkam} we demystify the recent formula for discrete iterates proposed by Monkam (2020) \cite{monkam2020}, which we generalize to fractional iterates in \Cref{s:monkamgen}. Finally, the corresponding new expressions for the \emph{iterative logarithm} are computed in \Cref{s:itlog}.

\section{Preliminaries}\label{s:preli}

In this section, we provide a concise overview of the key tools relevant to our study.

\subsection{Coefficients of linear operators over polynomials}\label{s:coeff}

An \emph{operator} is a linear endomorphism of $\C[x]$ that does not increase degree. It is often convenient to write the composition and application of an operator as an implicit product, so that $(U \circ V)(p)(x)$ is written more clearly as $UV p(x).$ For any operator, we define its \emph{coefficients}\footnote{The symbol for the coefficients should not be mistaken for the notation that is frequently used to represent $q$-\emph{binomial coefficients} \cite{kac2002}, as we adopt a different notation for that purpose later.} \cite{beauduin2024} by
\begin{equation}
    U x^n := \sum_{k=0}^n \coeff{n}{k}_U x^k.
\end{equation}
For example, if $U$ is the identity, denoted by $1,$ we write
\begin{equation}
    \coeff{n}{k}_1 = \delta_{n-k},
\end{equation}
where $\delta_n$ is the \emph{Kronecker delta} defined to be $1$ whenever $n$ is $0,$ and $0$ otherwise. Coefficient extraction is linear in the operator, that is, for two operators $U,V$ and $\lambda, \mu\in\C:$
\begin{equation}
    \coeff{n}{k}_{\lambda U + \mu V} = \lambda \coeff{n}{k}_U + \mu \coeff{n}{k}_V,
\end{equation}
whereas its interaction with composition can be worked out to be
\begin{equation}
    \coeff{n}{k}_{UV} = \sum_{r=k}^n \coeff{r}{k}_U \coeff{n}{r}_V.
\end{equation}
Both formulas readily generalize to finite sums/compositions: for $p\in\N,$ some scalars $(\lambda_i)_{0\leq i<p}$ and operators $(U_i)_{0\leq i< p}$
\begin{equation}\label{e:linear}
    \coeff{n}{k}_{\sum_{i=0}^{p-1} \lambda_i U_i} = \sum_{i=0}^{p-1} \lambda_i \coeff{n}{k}_{U_i},
\end{equation}
and
\begin{equation}\label{e:compos}
    \coeff{n}{k}_{\prod_{i=0}^{p-1} U_i} = \sum_{k=r_0\leq \ldots\leq r_p=n} \prod_{i=0}^{p-1} \coeff{r_{i+1}}{r_i}_{U_i}.
\end{equation}
The latter summation ranges over all nondecreasing maps $r : \bbra{0, p} \to \bbra{k, n},$ where $\bbra{a, b} := \{a, a+1, \ldots, b-1, b\},$ such that $r_0 = k$ and $r_p = n.$

\subsection{Bell polynomials}

In this section and thereafter, we consider $f$ to be an invertible (formal) power series. That is, it can be expressed as
\begin{equation}\label{e:f}
    f(t) := \sum_{n=1}^\infty \frac{t^n}{n!} a_n, \quad a_1 \neq 0,
\end{equation}
where $(a_n)_{n\ge1}$ is a sequence of numbers known as the \emph{exponential coefficients} of $f.$ The \emph{ordinary coefficients} are given by $q_n := a_n/n!.$ For later use, we define
\begin{equation}\label{e:q}
    q := f'(0) = a_1 = q_1.
\end{equation}
The \emph{exponential partial Bell polynomials} \cite[Chap.~3.3]{comtet1974} are defined by the following equation:
\begin{equation}
    \frac{f(t)^k}{k!} = \sum_{n=k}^\infty \frac{t^n}{n!} B_{n,k}(a_1, \ldots, a_{n-k+1}),
\end{equation}
i.e., they are the exponential coefficients of $f(t)^k/k!$ and are given explicitly by
\begin{equation}
    B_{n,k}(a_1, \ldots, a_{n-k+1}) = n! \sum_{\substack{1\ell_1 + \ldots + n\ell_n = n\\ \ell_1 + \ldots + \ell_n = k}} \prod_{p=1}^{n-k+1} \frac{a_p^{\ell_p}}{(p!)^{\ell_p} \ell_p!}.
\end{equation}
Similarly, the \emph{ordinary partial Bell polynomials} are the ordinary coefficients of $f(t)^k,$ that is,
\begin{equation}
    f(t)^k = \sum_{n=k}^\infty t^n \hat B_{n,k}(q_1, \ldots, q_{n-k+1}).
\end{equation}

The two kinds of Bell polynomials are related by
\begin{equation}\label{e:bells}
     B_{n,k}(a_1, \ldots, a_{n-k+1}) = \frac{n!}{k!} \hat B_{n,k}(q_1, \ldots, q_{n-k+1}).
\end{equation}

\emph{Jabotinsky matrices} $\textbf B(f)$ \cite{comtet1974,knuth1992} are infinite lower triangular matrices consisting of ordinary Bell polynomials, i.e., $\textbf B(f) = (\hat B_{n, k}(q_1, \dots, q_{n-k+1}))_{n, k\in\N}.$

\subsection{Fractional iteration}\label{s:fraciter}

Invertible formal power series form a group under composition. Thus, for $s\in\N,$ the iterate $f^s$ is defined as the $s$-fold self-composition of $f.$ For negative integers, setting $f^s := (f^{-1})^{-s}$ extends this definition so that the functional equation
\begin{equation}\label{e:iter}
    f^r \circ f^s = f^{r+s}
\end{equation}
holds for $r, s\in\Z := \{0, \pm 1, \pm 2, \dots\}.$

\emph{Fractional iteration} extends the parameter from integers to complex numbers: a family of formal power series $(f^s)_{s\in\C}$ is called a family of \emph{fractional iterates} of $f$ if $f^1 = f$ and \cref{e:iter} holds for all $r, s \in \C.$

Fractional iterates are not uniquely defined as such. The following lemma provides some intuition for the definition of a distinguished family of fractional iterates from the viewpoint of formal power series \cite{baker1962,labelle1980}.
\begin{lemma}\label[lemma]{l:itdef}
    If there exists a sequence of polynomials $(p_n(s))_{n>0}$ such that
    \begin{equation}\label{e:itdef}
        f^s(t) = \sum_{n=1}^\infty p_n(s) t^n,
    \end{equation}
    for $s\in\N,$ then the power series defined by \cref{e:itdef} with $s\in\C$ satisfy \cref{e:iter} for $r, s \in \C.$ Moreover, such an extension is unique among families whose coefficients are polynomial in $s.$
\end{lemma}

\begin{proof}
    For $r,s\in\N,$ equating coefficients in the equality \eqref{e:iter} yields polynomial identities in the variables $r$ and $s.$ Since these identities hold for all $r,s\in\N,$ they remain valid for all $r,s\in\C.$ Under the extension defined by \cref{e:itdef} with $s\in\C,$ the identities are equivalent to \cref{e:iter}. Two polynomials agreeing for all $s\in\N$ are equal; hence the extension is unique.
\end{proof}

When $f'(0) = q = 1,$ the hypotheses of \Cref{l:itdef} are always satisfied, as follows from \Cref{t:monkam,e:hbinom} below. A classical and illustrative example of fractional iterates in that case is given by the function $f_1(t) = t/(1-t),$ for which $f_1^s(t) = t/(1-st).$

When $q$ is not a root of unity, the coefficients instead take the form $P_n(q^s)$ for suitable polynomials $P_n,$ a result dating back to Collins (1836) \cite{collins1836}. The values $(q^s)_{s\in\N}$ are distinct, so the polynomial-identity argument used in the proof of \Cref{l:itdef} still applies. The function $f_q(t):=qt/(1-t)$ has fractional iterates
\begin{equation}\label{e:qiter}
    f_q^s(t)=\frac{q^s t}{1-\frac{q^s-1}{q-1} t},
\end{equation}
under a fixed determination $\log q$ of the complex logarithm, with the conventions
\[
q^s:=e^{s\log q},\quad
\log(q^r):=r\log q, \quad
(q^r)^s:=q^{rs}, \quad r\in\N,\, s\in\C,
\]
which we use throughout the remainder of this paper. Thus, \cref{e:qiter} reduces to the preceding example when $q\to 1.$  As will become clear later, the same compatibility with the limit $q\to1$ holds more generally.

These choices give the natural extensions of the discrete iterates used throughout this paper and agree with the fractional iterates considered in the literature, notably those of Écalle for $q>0$ in Theorem~5 of \cite{ecalle1971,ecalle1975}.

When $q$ is a nontrivial root of unity of order $m,$ the situation is more delicate: fractional iterates exist if and only if $f^m(t)=t$ \cite[p.~49]{babenko2013}. The family in \cref{e:qiter} satisfies this condition, since $q^m=1$ implies that $f_q^m(t)=t.$ Thus, outside the unitary case, we henceforth assume that $q$ is not a root of unity.

\subsection{Umbral calculus}\label{s:uc}

An invertible power series gives rise to a sequence of polynomials $(\phi_n(x))_{n\in\N}$ via the generating function \cite{beauduin2024}
\begin{equation}\label{e:basicdef}
    e^{xf(t)} = \sum_{n=0}^\infty \frac{t^n}{n!} \phi_n(x).
\end{equation}
The sequence, referred to as a \emph{basic sequence}, is characterized by the following properties \cite{beauduin2024}:
\begin{itemize}
    \item $(\phi_n(x))_{n\in\N}$ is a \emph{polynomial sequence}, that is, $\phi_n(x)$ is of degree $n.$ We shall write these with braces: $\{\phi_n(x)\}_{n\in\N}.$
    \item $\phi_0(x) = 1.$
    \item $\phi_n(0) = 0$ for $n>0.$
    \item $\{\phi_n(x)\}_{n\in\N}$ is a \emph{Sheffer sequence}, i.e., there exists a \emph{delta operator} $Q$ such that for all $n>0,$ $Q \phi_n(x) = n \phi_{n-1}(x).$
\end{itemize}

$Q$ is, in fact, unique and equal to $f^{-1}(D),$ where $D$ is the differentiation operator. Furthermore, each polynomial sequence $\{\phi_n(x)\}_{n\in\N}$ can be uniquely associated with an invertible linear operator $\phi,$ so that $\phi x^n = \phi_n(x).$ When $\{\phi_n(x)\}_{n\in\N}$ is a basic sequence, $\phi$ is called an \emph{umbral operator}. If $\psi$ is the umbral operator associated with the power series $g,$ then $\phi \psi$ is the umbral operator associated with $f \circ g.$ Consequently, the iterates $\phi^s,$ for $s\in\Z,$ are umbral operators for $f^s$ \cite{beauduin2024}.

It follows from \cref{e:basicdef} that the generating function of the coefficients of an umbral operator can be expressed by
\begin{equation}
    \frac{f(t)^k}{k!} = \sum_{n=k}^\infty \frac{t^n}{n!} \coeff{n}{k}_\phi.
\end{equation}
Therefore, looking back at \cref{e:f}, we see that $\coeff{n}{1}_\phi = a_n$ and that we can express the coefficient with Bell polynomials
\begin{equation}\label{e:coeffBell}
    \coeff{n}{k}_\phi = B_{n,k}(a_1, \ldots, a_{n-k+1}).
\end{equation}
Thanks to the additional properties established in \Cref{s:coeff}, these coefficients extend the notion of Jabotinsky matrices (up to the isomorphism \eqref{e:bells}) to the case where $\phi$ is not an umbral operator. The property
\[
\coeff{n}{k}_{\phi\psi} = \sum_{r=k}^n \coeff{n}{r}_\psi \coeff{r}{k}_\phi
\]
is the coefficient analogue of the fundamental composition rule for Jabotinsky matrices $\textbf B(f\circ g) = \textbf B(g) \textbf B(f)$ \cite[Thm.~A of Sec.~3.7]{comtet1974}. By contrast, there is no corresponding property for Jabotinsky matrices, since the map $f\mapsto\textbf B(f)$ is not linear. Operator coefficients remain meaningful for addition, even though $\phi+\psi$ is generally not an umbral operator, and satisfy $\coeff{n}{k}_{\phi+\psi}=\coeff{n}{k}_\phi+\coeff{n}{k}_\psi.$ This additional linearity is the key computational advantage used in our results.

The diagonal terms are given by
\begin{equation}\label{e:q^n}
    \coeff{n}{n}_\phi = q^n,
\end{equation}
and from the preceding discussion, for $s\in\Z,$ the following result is clear:
\begin{equation}\label{e:f^s}
    \frac{f^s(t)^k}{k!} = \sum_{n=k}^\infty \frac{t^n}{n!} \coeff{n}{k}_{\phi^s}.
\end{equation}

Using \cref{e:basicdef}, we define the fractional iterates $\phi^s$ of $\phi$ to be the umbral operators whose basic sequence is generated by $e^{xf^s(t)},$ for $s\in\C.$ Their coefficients are therefore given by \cref{e:f^s}. Although we defined $\phi^s$ merely as notation, it can also be realized as a fractional power of the linear operator $\phi$ \cite{beauduin2025} under the generalized exponentiation defined in \cite{beauduin2025a}.The goal of the subsequent sections is to give different expressions for the coefficient $\coeff{n}{k}_{\phi^s}.$

\section{Discrete iteration}\label{s:discrete}

This section focuses on iteration formulas specific to discrete iterations, that is, for $s\in\N.$ Although \Cref{s:fracit} mainly discusses formulas where $s$ is arbitrary, one should keep in mind that these are also applicable for $s\in\N.$ For example, \emph{Schröder's formula} \cite{schroeder1870}, the earliest known formula for discrete iteration, is discussed in \Cref{s:fracit}, as it naturally extends to fractional iteration.

\subsection{The matrix formula}

The first formula below can be interpreted as the expansion of the coefficients for the $s$-th power of the Jabotinsky matrix of $f,$ as in \cite[eq.~(2)]{jabotinsky1947}.
\begin{theorem}\label[theorem]{t:matrix}
    If $s\in\N,$ the coefficients of $f^s$ \eqref{e:f^s} are given by
    \begin{equation}\label{e:matrix}
        \coeff{n}{k}_{\phi^s} = \sum_{k=r_0\leq \ldots\leq r_s=n} \prod_{i=0}^{s-1} \coeff{r_{i+1}}{r_i}_\phi.
    \end{equation}
\end{theorem}

\begin{proof}
    It follows immediately from \cref{e:compos}.
\end{proof}

Replacing the coefficient with its formula in terms of Bell polynomials \eqref{e:coeffBell}, converting them into an equation in terms of ordinary Bell polynomials via \cref{e:bells}, and using a telescoping product yields
\begin{equation}\label{e:matrixord}
    \coeff{n}{k}_{\phi^s} = \frac{n!}{k!} \sum_{k=r_0\leq \ldots\leq r_s=n} \prod_{i=0}^{s-1} \hat B_{r_{i+1},r_i}(q_1, \ldots, q_{r_{i+1}-r_i+1}).
\end{equation}

\subsection{Monkam's formula}\label{s:monkam}

Monkam recently proposed a new expression for the discrete iterates of a power series \cite{monkam2020}. However, both the presentation of the formula and the proof leave its underlying structure somewhat unclear. We show that a particular set partition allows one to rederive his result directly from \Cref{t:matrix}, and that the formula can be expressed more transparently by means of a natural object: a symmetric polynomial.

\begin{definition}
    \emph{Complete homogeneous symmetric polynomials} of degree $k$ in $n$ variables can be defined by either of the two formulas below \textnormal{\cite[Sec.~7.5]{stanley1999}}
    \begin{equation}\label{e:chsp}
        h_k(x_1, \ldots, x_n) := \sum_{1\leq i_1\leq\ldots\leq i_k \leq n} x_{i_1} \cdots x_{i_k} = \sum_{\lambda_1 + \ldots + \lambda_n = k} x_1^{\lambda_1} \cdots x_n^{\lambda_n}.
    \end{equation}
    When $k$ is negative, $h_k = 0.$
\end{definition}

\begin{theorem}[Monkam]\label[theorem]{t:monkam}
    If $s\in\N,$ the coefficients of $f^s$ \eqref{e:f^s} are given by
    \begin{equation}\label{e:monkam}
        \coeff{n}{k}_{\phi^s} = \sum_{p=0}^{n-k} \sum_{k=r_0<\ldots<r_p=n} h_{s-p}(q^{r_0}, \ldots, q^{r_p}) \prod_{i=0}^{p-1} \coeff{r_{i+1}}{r_i}_\phi,
    \end{equation}
    where the upper bound on the first sum can also be $s$ or $\min(n-k,s),$ and
    \begin{equation}\label{e:h}
        h_{s-p}(q^{r_0}, \ldots, q^{r_p}) = \sum_{0\leq i_1\leq\ldots\leq i_{s-p} \leq p} q^{r_{i_1} + \ldots + r_{i_{s-p}}} = \sum_{\lambda_0 + \ldots + \lambda_p = s-p} q^{r_0\lambda_0 + \ldots + r_p\lambda_p}.
    \end{equation}
\end{theorem}

To prove \Cref{t:monkam}, we require the following lemma.

\begin{lemma}\label[lemma]{l:partition}
    The following partitioning equality holds:
    \begin{equation}\label{e:partition}
        \{(r_i)_{0\le i\le s}\ |\ k = r_0 \leq \ldots \leq r_s = n \} = \bigsqcup_{p=0}^s S_p,
    \end{equation}
    where the sets $(S_p)_{0\leq p\leq s}$ consist of functions $r:\bbra{0, s} \to \bbra{k, n}$ such that there exist a unique function $R:\bbra{0, p} \to \bbra{k, n},$ satisfying $k=R_0 < \ldots < R_p=n,$ and unique indices $0\leq i_1 \leq \ldots \leq i_{s-p}\leq p$ such that $r_i$ takes the $i$-th value of the ordering of the $(s+1)$-tuple $(R_0, \ldots, R_p, R_{i_1}, \ldots, R_{i_{s-p}}).$
\end{lemma}

In the partitioning equality, $p+1$ represents the number of distinct values taken by $r,$ and the indices $i_\alpha$ encode the extra copies of the repeated values of $r.$ For example, let $p = 2,$ $s = 5,$ $k=3,$ $n=6$ and
\begin{align*}
    (r_0, r_1, r_2, r_3, r_4, r_5) :={}& (3, 3, 4, 4, 4, 6) \\ = (R_0, R_0, R_1, R_1, R_1, R_2) ={}& (R_{i_1}, R_0, R_{i_2}, R_{i_3}, R_1, R_2),
\end{align*}
where $i_1 = 0$ and $i_2 = i_3 = 1.$

\begin{proof}
    The sets $S_p$ are pairwise disjoint because the number of distinct values taken by $r\in S_p$ is $p+1.$ To prove the equality between both sets, we proceed by double inclusion:
    \begin{itemize}
        \item[( $\supset$ )] This inclusion is trivial, as $r$ is nondecreasing.
        \item[( $\subset$ )] Let $E := \{r_0, \dots, r_s\}.$ The values of $r$ may contain duplicates, hence we may write $E = \{R_0, \dots, R_p\}$ for some $R_0 < \ldots < R_p$ and where $p+1$ is the cardinality of $E.$ This ordering uniquely determines the $R_i,$ and consequently there exist unique $s-p$ indices $i_1 \le\dots\le i_{s-p}$ such that $(r_0, \dots, r_s)$ is the ordering of $(R_0, \ldots, R_p, R_{i_1}, \ldots, R_{i_{s-p}}).$
    \end{itemize}
\end{proof}

\begin{proof}[Proof of \Cref{t:monkam}]
    We can now break up the sum of \Cref{t:matrix} according to the partition of \Cref{l:partition}. Furthermore, denote $A_r := \cbra{i : r_i = r_{i+1}}$ and $B_r := \cbra{i : r_i < r_{i+1}},$ which form a partition of $\bbra{0, s-1}.$
    \[
    \begin{gathered}
        \coeff{n}{k}_{\phi^s} = \sum_{k=r_0\leq \ldots\leq r_s=n} \prod_{i=0}^{s-1} \coeff{r_{i+1}}{r_i}_\phi \\
        = \sum_{p=0}^s \sum_{r\in S_p} \prod_{i\in A_r\sqcup B_r} \coeff{r_{i+1}}{r_i}_\phi = \sum_{p=0}^s \sum_{r\in S_p} \prod_{i\in A_r} \coeff{r_i}{r_i}_\phi \prod_{i\in B_r} \coeff{r_{i+1}}{r_i}_\phi.
    \end{gathered}
    \]
    We rename the variables according to \Cref{l:partition} and use \cref{e:q^n,e:chsp}:
    \begin{align*}
        \coeff{n}{k}_{\phi^s}
        &= \sum_{p=0}^s \sum_{\substack{k=R_0<\ldots<R_p=n \\ 0 \leq i_1 \leq \ldots \leq i_{s-p} \leq p}} \prod_{a=1}^{s-p} \coeff{R_{i_a}}{R_{i_a}}_\phi \prod_{b=0}^{p-1} \coeff{R_{b+1}}{R_b}_\phi \\
        &= \sum_{p=0}^s \sum_{k=R_0<\ldots<R_p=n} \pa{\sum_{0 \leq i_1 \leq \ldots \leq i_{s-p} \leq p} \prod_{a=1}^{s-p} q^{R_{i_a}}} \prod_{b=0}^{p-1} \coeff{R_{b+1}}{R_b}_\phi \\
        &= \sum_{p=0}^s \sum_{k=R_0<\ldots<R_p=n} h_{s-p}(q^{R_0}, \ldots, q^{R_p}) \prod_{b=0}^{p-1} \coeff{R_{b+1}}{R_b}_\phi.
    \end{align*}
    Moreover, if $p > n-k,$ given that $R_0, \ldots, R_p$ are distinct integers lying between $k$ and $n,$ the summation given above ranges over an empty set. The upper bound for $p$ may be taken as $n-k,$ as $s$ or, more sharply, the minimum of these two values.
\end{proof}

In his work, Monkam \cite{monkam2020} presented an alternative formulation of the complete homogeneous symmetric polynomial. To obtain his version, we use the second formula of \cref{e:h}. Since $r_0 \lambda_0 = k(s-p - \lambda_1 - \ldots - \lambda_p),$ we have
\begin{align}\label{e:hmonkam}
    &h_{s-p}(q^{r_0}, \ldots, q^{r_p}) = \sum_{\lambda_0 + \ldots + \lambda_p = s-p} q^{r_0\lambda_0 + \ldots + r_p\lambda_p} \nonumber\\
    ={}& \sum_{\lambda_1 = 0}^{s-p} \sum_{\lambda_2 = 0}^{s - p - \lambda_1} \dots \sum_{\lambda_p=0}^{s - p - \lambda_1 - \ldots - \lambda_{p-1}}  q^{k(s-p - \lambda_1 - \ldots - \lambda_p)} q^{r_1 \lambda_1 + \ldots + r_p\lambda_p} \nonumber\\
    ={}& q^{k(s-p)}\sum_{\lambda_1 = 0}^{s-p} q^{(r_1 - k)\lambda_1} \sum_{\lambda_2 = 0}^{s - \lambda_1 - p} q^{(r_2 - k)\lambda_2} \dots \sum_{\lambda_p=0}^{s - \lambda_1 - \ldots - \lambda_{p-1} - p} q^{(r_p - k)\lambda_p}.
\end{align}
\Cref{e:hmonkam} with $k=1$ appears in \cite[Thm.~4.1]{monkam2020}. Monkam had also noticed that the case $q=1$ simplifies to Schröder's formula (\Cref{t:schroder} of \Cref{s:frac1}). This can be proved with the identity \cite{stanley1999}
\begin{equation}\label{e:hbinom}
    h_{s-p}(\underbrace{1, \ldots, 1}_{p+1 \text{ times}}) = \binom{s}{p} = \frac{s(s-1) \cdots (s-p+1)}{p!}.
\end{equation}

\section{Fractional iteration}\label{s:fracit}

By combining a remarkable result from Baker \cite{baker1962} with another, independently found by Écalle \cite{ecalle1973c} and Liverpool \cite{liverpool1975}, we know that for a formal power series $f,$ satisfying $f'(0) = 1,$ the set of values of $s$ for which $f^s$ possesses a positive \emph{radius of convergence} is either $\C$ or $\lambda\Z,$ for some $\lambda\in\C,$ with $\lambda=0$ corresponding to the degenerate case. See the survey \cite{babenko2013}. In the former case, we say that $f$ is \emph{embeddable}. However, it is also known that most power series encountered in practice are not embeddable \cite{baker1967} and therefore that for most convergent power series, only their discrete iterates converge (the case $\lambda=1$). The function $f(t) = t/(1-t)$ and its iterates $f^s(t) = t/(1-st)$ are examples of embeddable power series, also known for being the only ones meromorphic over the entire complex plane under $f'(0) = 1$ \cite{baker1964}. See the surveys \cite{scheinberg1970,babenko2013} for further discussions.

In this section, we present three general formulas for the coefficients of the fractional iterates of an invertible power series. We begin by revisiting the classical special case $f'(0)=1$ within our framework. This simpler setting sharpens our understanding and prepares the way for a more general treatment, where tools from $q$-calculus are introduced and their relevance motivated.

\subsection{Unitary case \texorpdfstring{$f'(0)=1$}{f'(0)=1}}\label{s:frac1}

The earliest known formula for iterating power series was introduced by Schröder in 1870 \cite{schroeder1870}, building on an observation of Cayley \cite{cayley1860}. Bennett recognized its natural extension to arbitrary complex iteration order \cite{bennett1915}, while Jabotinsky subsequently generalized part of Schröder's result from $k=1$ to arbitrary $k$ using Bennett's matrix representation \cite{jabotinsky1947,jabotinsky1963}. We complete this extension by generalizing the accompanying coefficient formula to arbitrary $k.$

\begin{theorem}[Schröder]\label[theorem]{t:schroder}
    If $q = 1$ and $s\in\C,$ the coefficients of $f^s$ \eqref{e:f^s} are given by
    \begin{equation}\label{e:schroder}
        \coeff{n}{k}_{\phi^s} = \sum_{p=0}^{n-k} \binom{s}{p} \coeff{n}{k}_{(\phi-1)^p},
    \end{equation}
    with
    \begin{equation}\label{e:coeff0}
        \coeff{n}{k}_{(\phi-1)^p} = \sum_{k=r_0<\ldots<r_p=n} \prod_{i=0}^{p-1} \coeff{r_{i+1}}{r_i}_\phi,
    \end{equation}
    which vanishes for $p > n-k.$
\end{theorem}

For $k=1,$ \eqref{e:schroder} and \eqref{e:coeff0} reduce respectively to Schröder's \cite[eqs.~(38b) and (40)]{schroeder1870}. Jabotinsky extended the former to arbitrary $k,$ while \eqref{e:coeff0} provides the corresponding extension of the latter.

\begin{proof}
By the realization of $\phi^s$ as a fractional power of the linear operator $\phi$ \cite{beauduin2025} and the linearity of the coefficients \eqref{e:linear}:
\begin{equation}\label{e:finitesum}
    \coeff{n}{k}_{\phi^s} = \coeff{n}{k}_{\sum_{p=0}^\infty \binom{s}{p} (\phi-1)^p} = \sum_{p=0}^\infty \binom{s}{p} \coeff{n}{k}_{(\phi-1)^p}.
\end{equation}
To justify the infinite linearity, we need to verify that only finitely many terms in \cref{e:finitesum} are nonzero. Indeed, by \cref{e:compos}
\begin{equation}\label{e:sum}
    \coeff{n}{k}_{(\phi-1)^p} = \sum_{k=r_0\leq\ldots\leq r_p=n} \prod_{i=0}^{p-1} \coeff{r_{i+1}}{r_i}_{\phi-1},
\end{equation}
and by \cref{e:q^n}, $\coeff{n}{n}_{\phi-1} = 0,$ thus we can remove zeros from the sum in \cref{e:sum} by tightening the restrictions of the sum from inequalities to strict ones:
\begin{equation}\label{e:simplif}
    \coeff{n}{k}_{(\phi-1)^p} = \sum_{k=r_0<\ldots<r_p=n} \prod_{i=0}^{p-1} \coeff{r_{i+1}}{r_i}_\phi.
\end{equation}
Moreover, by the same argument used in the proof of \Cref{t:monkam}, \cref{e:simplif} vanishes whenever $p > n-k.$ Hence, the series in \cref{e:finitesum} is indeed finite and terminates at $n-k.$
\end{proof}

Jabotinsky also provided a convenient alternative formulation \cite{jabotinsky1963}, expressing fractional iterates in terms of discrete iterates, which can be computed via \Cref{t:matrix}.
\begin{theorem}[Jabotinsky]\label[theorem]{t:jabo}
    If $q = 1$ and $s\in\C,$ the coefficients of $f^s$ \eqref{e:f^s} are given by
    \[
    \coeff{n}{k}_{\phi^s} = \sum_{p=0}^{n-k} \coeff{n}{k}_{\phi^p} \binom{s}{p} \binom{n-k-s}{n-k-p}.
    \]
\end{theorem}

Jabotinsky established the result by first proving it for integer values of $s,$ and then extending it to arbitrary $s$ using an argument based on the equality of polynomials on the integers (see also \cite{knuth1992}). Our approach, by contrast, yields a direct proof.

\begin{proof}
First, we use the binomial expansion and \cref{e:linear} in the equation
\[
\coeff{n}{k}_{(\phi-1)^p} = \coeff{n}{k}_{\sum_{\ell=0}^p \binom{p}{\ell} (-1)^{p-\ell} \phi^\ell} = \sum_{\ell=0}^p \binom{p}{\ell} (-1)^{p-\ell} \coeff{n}{k}_{ \phi^\ell},
\]
which we substitute in \cref{e:schroder}. Exchanging sums and applying standard binomial identities \cite[Tab.~174]{graham1994} yield the result.
\end{proof}

More specifically, Jabotinsky derived the following formula:
\begin{equation}
    \coeff{n}{k}_{\phi^s} = \sum_{p=0}^{n-k} \coeff{n}{k}_{\phi^p} \binom{s}{p} \binom{s-1-p}{n-k-p} (-1)^{n-k-p},
\end{equation}
equivalent to the theorem through the identity
\begin{equation}\label{e:binom-}
    \binom{a}{b} (-1)^b = \binom{b-1-a}{b}.
\end{equation}

An alternative expression of interest can be derived from Jabotinsky's. Specifically, by substituting $n$ with $n+k$ and employing different combinatorial identities, it is possible to isolate most of the $s$-dependence in front:
\begin{equation}\label{e:extract}
    \coeff{n+k}{k}_{\phi^s} = \binom{s}{n} \sum_{p=0}^n \coeff{n+k}{k}_{\phi^p}  \binom{n}{p} \frac{s-n}{s-p} (-1)^{n-p},
\end{equation}
valid for values of $s\notin\bbra{0,n}.$ This version has potential for the asymptotic study of the coefficient as $n$ or $s$ approaches infinity.

More generally, Jabotinsky also obtained an explicit formula for arbitrary complex powers of the iterates \cite[Thm.~I]{jabotinsky1963}.

\subsection{\texorpdfstring{$q$}{q}-calculus}

The requirement that $f'(0) = q = 1$ is crucial to prove \cref{e:simplif}, and it does not seem that the argument can be altered to cover the scenario where $q \neq 1.$ Therefore, an entirely different analysis is necessary. We aim to motivate the application of the \emph{$q$-calculus} for this purpose. The expressions for the first few coefficients seem to suggest this approach. When $s\in\N,$ these can be computed via Taylor's formula. By repeatedly applying the chain rule, we find for the first coefficient:
\begin{equation}\label{e:first}
    \odv{f^s(t)}{t}_{t=0} = \rbar{\prod_{i=0}^{s-1} f'(f^i(t))}_{t=0} = q^s.
\end{equation}
To derive the next term, we employ logarithmic differentiation on the product of \cref{e:first} to obtain
\begin{align}
    \odv[order=2]{f^s(t)}{t}_{t=0}
    &= \rbar{\pa{\prod_{i=0}^{s-1} f'(f^i(t))} \pa{\sum_{i=0}^{s-1} \frac{f''(f^i(t))}{f'(f^i(t))} \prod_{j=0}^{i-1} f'(f^j(t))}}_{t=0} \\
    &= q^s \sum_{i=0}^{s-1} \frac{a_2}{q} q^i = [s]_q q^{s-1} a_2, \label{e:second}
\end{align}
where $[s]_q := 1 + q + \ldots + q^{s-1}$ is a \emph{$q$-integer}. This suggests that $q$-calculus is involved in the general formulation of the coefficients; this is also visible in \cref{e:qiter} since
\begin{equation}\label{e:geom}
    [s]_q = \frac{q^s-1}{q-1}
\end{equation}
for $s\in\N$ and $q\neq1,$ by the geometric sum formula.

In what follows, we present the basic elements of the theory that we will need in the next section to derive the formula. These elements will be directly taken from the Kac--Cheung monograph \cite{kac2002} on the subject, starting with the \emph{$q$-derivative}:
\begin{equation}
    D_q f(t) = \frac{f(qt) - f(t)}{qt - t},
\end{equation}
for $q\neq1.$ Observe that when $q\to1,$ the standard derivative is recovered. This limiting behavior is a typical aspect of $q$-calculus, that is, when allowing $q$ to tend to $1,$ we recover the ``classical'' calculus.

The previously defined $q$-integers satisfy a relationship with the $q$-derivative:
\begin{equation}
    D_q t^n = [n]_q t^{n-1},
\end{equation}
which follows from \cref{e:geom}. After defining the \emph{$q$-factorial} $[n]_q! := [n]_q [n-1]_q \cdots [1]_q,$ the above equation generalizes, for $p \leq n,$ to
\begin{equation}
    D_q^p t^n = \frac{[n]_q!}{[n-p]_q!} t^{n-p}.
\end{equation}
Then, the \emph{$q$-binomial coefficients} are defined by
\begin{equation}
    \qbinom{n}{p} := \frac{[n]_q!}{[p]_q! [n-p]_q!},
\end{equation}
and readily extended to complex upper argument as follows:
\begin{equation}\label{e:qbinom}
    \qbinom{s}{p} := \frac{[s]_q [s-1]_q \cdots [s-p+1]_q}{[p]_q!} = \frac{(q^s-1) (q^{s-1} - 1) \cdots (q^{s-p+1} - 1)}{(q^p-1) (q^{p-1} - 1) \cdots (q - 1)},
\end{equation}
where \emph{$q$-complex} numbers are naturally defined by \cref{e:geom} and $[s]_1 := s$ for $s\in\C.$ \Cref{e:qbinom} makes clear why $q$ is assumed not to be a root of unity. These coefficients are useful to express the expansion of the $q$-analogue of a binomial, namely \emph{Gauss's formula}:
\begin{equation}\label{e:gauss}
    (a + b)_q^n := \prod_{i=0}^{n-1} (a + b q^i) = \sum_{\ell=0}^n \qbinom{n}{\ell} q^{\binom\ell{2}} a^{n-\ell} b^\ell.
\end{equation}
Finally, the $q$-analogue of Taylor's formula for an arbitrary constant $c$ is
\begin{equation}\label{e:qtaylor}
    g(t) = \sum_{p=0}^\infty D_q^p g(c) \frac{(t-c)_q^p}{[p]_q!},
\end{equation}
which will only be considered in the formal sense \cite[Sec.~8]{kac2002}.

\subsection{Extension of Schröder's formula}\label{s:schrodergen}

Using the $q$-Taylor formula \eqref{e:qtaylor}, we can now extend Schröder's formula (\Cref{t:schroder}) to the nonunitary case.

\begin{theorem}\label[theorem]{t:gen}
    Suppose that $q=f'(0)$ is not a root of unity. If $s\in\C,$ the coefficients of $f^s$ \eqref{e:f^s} are given by
    \begin{equation}\label{e:gen}
        \coeff{n}{k}_{\phi^s} = \sum_{p=0}^{n-k} \qbinom{s}{p} q^{k(s-p)} \coeff{n}{k}_{(\phi-q^k)_q^p},
    \end{equation}
    with
    \begin{equation}\label{e:coeff1}
        \coeff{n}{k}_{(\phi-q^k)_q^p} = \sum_{k=r_0 \leq \ldots \leq r_p = n} \prod_{i=0}^{p-1} \pa{\coeff{r_{i+1}}{r_i}_\phi - q^{i+k} \delta_{r_{i+1}-r_i}},
    \end{equation}
    which vanishes for $p > n-k.$
\end{theorem}

Note that, with $q \to 1,$ the same argument used for \cref{e:simplif} can be applied to recover \cref{e:coeff0} from \cref{e:coeff1}.

\begin{proof}
    In proving Schröder's formula, we used the generalized binomial expansion of the fractional power $\phi^s.$ To obtain an analogous finite sum in the general case, we first take $s\in\N$ and apply the $q$-Taylor formula \eqref{e:qtaylor} to the monomial $x^s,$ with $x=\phi$:
    \begin{equation}
        \phi^s = \sum_{p=0}^\infty \qbinom{s}{p} c^{s-p} (\phi-c)_q^p.
    \end{equation}
    The resulting coefficient identity will then extend to $s\in\C,$ since both sides are polynomial in $q^s$ and agree for all $s\in\N.$

    When $q=1,$ it is natural to take $c = 1.$ In the general case, a choice $c = c(q)$ that satisfies $c(1) = 1$ is thus necessary. The naive choice $c(q) = 1$ still produces a finite sum, but its upper limit becomes $n$ rather than the sharper bound $n-k$ obtained in \Cref{t:schroder}. To recover the optimal upper bound $n-k,$ the appropriate choice is $c(q) = q^k.$ Therefore, by \cref{e:linear}
    \begin{equation}\label{e:infinitesum}
        \coeff{n}{k}_{\phi^s} = \sum_{p=0}^\infty \qbinom{s}{p} q^{k(s-p)} \coeff{n}{k}_{(\phi-q^k)_q^p}.
    \end{equation}
    Now, to show that the sum is finite (and ranging up to $n-k$), we first expand the coefficient in \cref{e:infinitesum}. By \cref{e:compos}
    \begin{align}
        \coeff{n}{k}_{(\phi-q^k)_q^p}
        &= \coeff{n}{k}_{\prod_{i=0}^{p-1} (\phi - q^{i+k})} = \sum_{k=r_0 \leq \ldots \leq r_p = n} \prod_{i=0}^{p-1} \coeff{r_{i+1}}{r_i}_{\phi - q^{i+k}} \nonumber\\
        &= \sum_{k=r_0 \leq \ldots \leq r_p = n} \prod_{i=0}^{p-1} \pa{\coeff{r_{i+1}}{r_i}_{\phi} - q^{i+k} \delta_{r_{i+1}-r_i}}. \label{a:prod}
    \end{align}
    Now we show that the expression above vanishes whenever $p>n-k.$ In this situation, the integers $R_i := r_{i-k},$ for $i\in\bbra{k, n+1},$ where the $r_i$ are the summation indices, satisfy the inequalities
    \begin{equation}\label{e:ineq}
        k= R_k \leq \ldots \leq R_{n+1} \leq n.
    \end{equation}
    $R$ admits at least one fixed point $k.$ Let $a\in\bbra{k, n}$ denote the largest fixed point. Then, by definition of $a,$ the interval $\bbra{a+1, n+1}$ contains no fixed points of $R.$

    Assume, for contradiction, that $R_{a+1} \geq a+1.$ Since $a+1$ is not a fixed point, we must have $R_{a+1} > a+1.$ As $R$ is nondecreasing, it follows that $R_{a+2} \geq R_{a+1} \geq a+2.$ Iterating this argument yields $R_{n+1} \geq n+1,$ contradicting \cref{e:ineq}. Hence, \[a = R_a \leq R_{a+1} < a+1,\] which forces $R_a = R_{a+1} = a.$ Translating this to $r,$ set $\alpha = a - k \in \bbra{0, n-k}.$ Then, $r_\alpha = r_{\alpha+1} = \alpha + k.$ The $\alpha$-th factor in the product of \cref{a:prod} equals
   \[
   \coeff{r_{\alpha+1}}{r_\alpha}_\phi - q^{\alpha + k} \delta_{r_{\alpha+1}-r_\alpha} = \coeff{\alpha + k}{\alpha + k}_\phi - q^{\alpha + k} = 0,
   \]
   by \cref{e:q^n}. Therefore, the product itself is also $0.$ This shows that the summands of \cref{e:infinitesum} vanish at the indices $p > n-k.$
\end{proof}

\subsection{Extension of Jabotinsky's formula}\label{s:jabogen}

The $q$-binomial $(\phi-q^k)_q^p$ appearing in \cref{e:gen} may be used in either of the two forms given by \cref{e:gauss}. Retaining its product form yields the extension of Schröder's formula obtained above. Expanding it as a sum of powers of $\phi$ instead leads to the following extension of Jabotinsky's formula (\Cref{t:jabo}).

\begin{theorem}[Tambs Lyche]\label[theorem]{t:tambs}
    Suppose that $q=f'(0)$ is not a root of unity. If $s\in\C,$ the coefficients of $f^s$ \eqref{e:f^s} are given by
    \begin{equation}
        \coeff{n}{k}_{\phi^s} = \sum_{p=0}^{n-k} \coeff{n}{k}_{\phi^p} \qbinom{s}{p} \qbinom{n-k-s}{n-k-p} q^{(n-p)(s-p)}.
    \end{equation}
\end{theorem}

\begin{proof}
    Similarly to the proof of \Cref{t:jabo}, we expand the following using \eqref{e:gauss}:
    \[
    \coeff{n}{k}_{(\phi-q^k)_q^p} =  \coeff{n}{k}_{\sum_{\ell=0}^p \qbinom{p}{\ell} q^{\binom\ell{2}} \phi^{p-\ell} (-q^k)^\ell} = \sum_{\ell=0}^p \coeff{n}{k}_{\phi^{p-\ell}} \qbinom{p}{\ell} q^{\binom{\ell}{2} + k \ell}  (-1)^\ell.
    \]
    We substitute the above equation into the formula \eqref{e:gen} and swap sums:
    \begin{align*}
        \coeff{n}{k}_{\phi^s}
        &= \sum_{p=0}^{n-k} \qbinom{s}{p} \sum_{\ell=0}^p \coeff{n}{k}_{\phi^\ell} \qbinom{p}{\ell} q^{\binom{p-\ell}{2} + k (s - \ell)}  (-1)^{p-\ell} \\
        &= \sum_{\ell=0}^{n-k} \coeff{n}{k}_{\phi^\ell} q^{k(s - \ell)} \sum_{p=\ell}^{n-k} \qbinom{s}{p} \qbinom{p}{\ell} q^{\binom{p-\ell}{2}}  (-1)^{p-\ell}.
    \end{align*}
    We then use the combinatorial identity which follows from the definition of the $q$-binomial coefficient
    \[
    \qbinom{s}{p} \qbinom{p}{\ell} = \qbinom{s}{\ell} \qbinom{s-\ell}{p-\ell},
    \]
    and reindex the sum:
    \begin{align*}
        \coeff{n}{k}_{\phi^s}
        &= \sum_{\ell=0}^{n-k} \coeff{n}{k}_{\phi^\ell} \qbinom{s}{\ell} q^{k(s - \ell)} \sum_{p=\ell}^{n-k} \qbinom{s-\ell}{p-\ell} q^{\binom{p-\ell}{2}}  (-1)^{p-\ell} \\
        &= \sum_{\ell=0}^{n-k} \coeff{n}{k}_{\phi^\ell} \qbinom{s}{\ell} q^{k(s - \ell)} \sum_{p=0}^{n-k-\ell} \qbinom{s-\ell}{p} q^{\binom{p}{2}}  (-1)^p.
    \end{align*}
   The theorem follows by applying the identity \cite[Id.~3.9 with $q=1$ and $m=i=0$]{shattuck2011}
    \begin{equation}
        \sum_{p=0}^{b} \qbinom{a}{p} q^{\binom{p}{2}}  (-1)^p = \qbinom{b-a}{b} q^{ba}.
    \end{equation}
\end{proof}
The case $k=1$ was identified by Lavoie and Tremblay \cite{lavoie1981}, who noted that it is a direct reformulation of a formula discovered much earlier by Tambs Lyche in 1927 \cite{tambslyche1927}. We can now extend this result to arbitrary natural values of $k$:
\begin{equation}\label{e:lavoie1981}
    \coeff{n}{k}_{\phi^s} = \sum_{p=0}^{n-k} \coeff{n}{k}_{\phi^p} \qbinom{s}{p} \qbinom{s-p-1}{n-k-p} q^{k(s-p) + \binom{n-k-p+1}{2}} (-1)^{n-k-p}.
\end{equation}
Indeed, it follows from the identity
\begin{equation}
    \qbinom{a}{b} (-1)^b =  \qbinom{b-a-1}{b} q^{ba-\binom{b}{2}},
\end{equation}
$q$-analogue of \cref{e:binom-}.

Tambs Lyche proved the formula for positive integers $s,$ and, like Jabotinsky, he asserted its validity for all $s$ real \cite{tambslyche1927,tambslyche1930}, and applied it for $s = -1$ (inversion).

As we did in \cref{e:extract}, we can factor out most of the dependence on $s$ to arrive at the equivalent statement
\begin{equation}\label{e:qextract}
    \coeff{n+k}{k}_{\phi^s} = \qbinom{s}{n} \sum_{p=0}^n \coeff{n+k}{k}_{\phi^p}  \qbinom{n}{p} \frac{[s]_q-[n]_q}{[s]_q-[p]_q} q^{k(s-p) + \binom{n-p}{2}} (-1)^{n-p},
\end{equation}
when $[s]_q \notin \{[0]_q, \dots, [n]_q\},$ which could be further simplified with the equality \begin{center}
$\frac{[a]_q-[b]_q}{[c]_q-[d]_q} = \frac{q^a - q^b}{q^c - q^d}.$
\end{center}

\subsection{Extension of Monkam's formula}\label{s:monkamgen}

In this section, we generalize Monkam's formula (\Cref{t:monkam}) to fractional iteration. To this end, we introduce the \emph{divided differences} \cite[Chap.~1]{milne-thomson1933} of a function $g$ defined by the base case $g[x_i] := g(x_i)$ and the recurrence relation
\begin{equation}
    g[x_i, \dots, x_j] := \frac{g[x_{i+1}, \dots, x_j] - g[x_i, \dots, x_{j-1}]}{x_j - x_i}.
\end{equation}
They are given explicitly by the equation \cite[\S 1.31]{milne-thomson1933}
\begin{equation}\label{e:fdformula}
    g[x_0, \dots, x_p] = \sum_{i=0}^p \frac{g(x_i)}{\prod_{j\neq i} (x_i - x_j)}.
\end{equation}

In \Cref{t:monkam}, we used the complete homogeneous symmetric polynomials, which are related to divided differences by the following equation:
\begin{equation}\label{e:symfd}
    h_{s-p}(x_0, \ldots, x_p) = [x_0, \dots, x_p]^s,
\end{equation}
where $s\in\N$ and $[x_0, \dots, x_p]^s$ represents $g[x_0, \dots, x_p]$ when $g(x) = x^s.$ \Cref{e:symfd} can be proved via \emph{Sylvester's identity} \cite{nica2023}
\begin{equation}
    h_{s-p}(x_0, \ldots, x_p) = \sum_{i=0}^p \frac{x_i^s}{\prod_{j\neq i} (x_i - x_j)}
\end{equation}
and \cref{e:fdformula}. When $s$ is not an integer, the symmetric polynomial cannot be used, since it is defined only for integer indices. This motivates the use of divided differences instead, which naturally leads to the following generalization.

\begin{theorem}\label[theorem]{t:monkamgen}
    Suppose that $q=f'(0)$ is not a root of unity. If $s\in\C,$ the coefficients of $f^s$ \eqref{e:f^s} are given by
    \begin{equation}\label{e:gen2}
        \coeff{n}{k}_{\phi^s} = \sum_{p=0}^{n-k}  \sum_{k=r_0<\ldots<r_p=n} [q^{r_0}, \dots, q^{r_p}]^s \prod_{i=0}^{p-1} \coeff{r_{i+1}}{r_i}_\phi.
    \end{equation}
\end{theorem}

At $p=1,$ we have $r_0 = k$ and $r_1 = n,$ therefore
\begin{equation}\label{e:p=1'}
   [q^{r_0}, q^{r_1}]^s = \frac{q^{k s} - q^{n s}}{q^k - q^n} = q^{k(s-1)} [s]_{q^{n-k}},
\end{equation}
where $[s]_{q^{n-k}}$ is a $q^{n-k}$-complex number. At $p=n-k,$ the conditions of the summation force $r_i = k + i$ for $i\in\bbra{0, p},$ thus we have only one term and the divided differences are given by
\begin{equation}\label{e:p=n-k}
    [q^k, q^{k+1}, \dots, q^n]^s = \qbinom{s}{n-k} q^{k(s-n+k)}.
\end{equation}

The linearity of the divided differences enables a simple proof for \Cref{t:monkamgen}.
\begin{proof}[Proof of \Cref{t:monkamgen}]
    By \Cref{t:gen}, Gauss's formula \eqref{e:gauss} and Monkam's formula \eqref{e:monkam},
    \begin{align*}
        \coeff{n}{k}_{\phi^s}
        ={}& \sum_{j=0}^{n-k}\qbinom{s}{j}q^{k(s-j)} \sum_{\ell=0}^j\qbinom{j}{\ell}q^{\binom{j-\ell}{2}+k(j-\ell)}(-1)^{j-\ell}\\
        &\quad{}\times\sum_{p=0}^{n-k}\sum_{k=r_0<\ldots<r_p=n}
        h_{\ell-p}(q^{r_0},\ldots,q^{r_p}) \prod_{i=0}^{p-1}\coeff{r_{i+1}}{r_i}_\phi.
    \end{align*}
    After exchanging the order of the sums, \cref{e:symfd} and the linearity of divided differences give
    \[
    \coeff{n}{k}_{\phi^s} = \sum_{p=0}^{n-k}\sum_{k=r_0<\ldots<r_p=n} P[q^{r_0},\ldots,q^{r_p}] \prod_{i=0}^{p-1}\coeff{r_{i+1}}{r_i}_\phi,
    \]
    where
    \begin{equation}
        P(x) := \sum_{j=0}^{n-k}\qbinom{s}{j}q^{k(s-j)}(x-q^k)_q^j.
    \end{equation}
    At $x=q^{r_i},$ all terms with $j>r_i-k$ vanish. Hence,
    \begin{equation}
        P(q^{r_i}) = \sum_{j=0}^\infty \qbinom{s}{j}q^{k(s-j)}(q^{r_i}-q^k)_q^j = q^{r_i s},
    \end{equation}
    by \cref{e:qtaylor}. Consequently, $P[q^{r_0},\ldots,q^{r_p}]=[q^{r_0},\ldots,q^{r_p}]^s,$ which gives \cref{e:gen2}.
\end{proof}

\subsection{Computation of the first few terms}

In this section, we compute the first few terms of the general formula, for $k=1$ and $n=1, 2, 3, 4$ from \Cref{t:monkamgen}. The first coefficient is
\begin{equation}\label{e:n=1}
    \coeff{1}{1}_{\phi^s} = \sum_{1=r_0<\ldots<r_0=1} [q^{r_0}]^s \prod_{i=0}^{-1} \coeff{r_{i+1}}{r_i}_\phi = q^s,
\end{equation}
in accordance with \cref{e:first}. For the second coefficient, at $p=0,$ we have $1 = r_0 = 2,$ so the sum is empty. At $p=1,$ we have $r_0 = 1$ and $r_1 = 2$ so, using \cref{e:p=1'}
\begin{equation}\label{e:n=2}
    \coeff{2}{1}_{\phi^s} = \sum_{1=r_0 < r_1=2} [q^{r_0}, q^{r_1}]^s \coeff{r_1}{r_0}_\phi = q^{s-1} [s]_q a_2
\end{equation}
in accordance with \cref{e:second}. At $n=3,$ the summands $p=0$ and $p=1$ have a similar treatment. For $p=2,$ according to \cref{e:p=n-k}
\begin{align*}
    \sum_{1=r_0<r_1<r_2=3} [q^{r_0}, q^{r_1}, q^{r_2}]^s \prod_{i=0}^1 \coeff{r_{i+1}}{r_i}_\phi
    &= [q, q^2, q^3]^s \coeff{2}{1}_\phi \coeff{3}{2}_\phi \\
    &= \qbinom{s}{2} q^{s-2} a_2 (3a_1a_2) \\
    &= 3\qbinom{s}{2} q^{s-1} a_2^2,
\end{align*}
where $(a_n)_{n\ge1}$ are the exponential coefficients of $f$ \eqref{e:f}, hence
\begin{equation}\label{e:n=3}
    \coeff{3}{1}_{\phi^s} = q^{s-1} [s]_{q^2} a_3 + 3\qbinom{s}{2} q^{s-1} a_2^2.
\end{equation}

For $n=4,$ at $p=1$ we have $q^{s-1} [s]_{q^3} a_4$ and at $p=2$ we get two terms:
\begin{align*}
    \sum_{1=r_0<r_1<r_2=4} & [q^{r_0}, q^{r_1}, q^{r_2}]^s \prod_{i=0}^1 \coeff{r_{i+1}}{r_i}_\phi \\
    ={}& [q, q^2, q^4]^s \coeff{2}{1}_\phi \coeff{4}{2}_\phi + [q, q^3, q^4]^s \coeff{3}{1}_\phi \coeff{4}{3}_\phi \\
    ={}& \frac{[q^2, q^4]^s - [q, q^2]^s}{q^4 - q} a_2 (3a_2^2 + 4a_1 a_3) + \frac{[q^3, q^4]^s - [q, q^3]^s}{q^4 - q} a_3 (6a_1^2 a_2) \\
    ={}& \frac{\frac{q^{4s} - q^{2s}}{q^4 - q^2} - \frac{q^{2s} - q^s}{q^2 - q}}{q^4 - q} a_2 (3a_2^2 + 4q a_3) + \frac{\frac{q^{4s} - q^{3s}}{q^4 - q^3} - \frac{q^{3s} - q^s}{q^3 - q}}{q^4 - q} 6 q^2 a_2 a_3
\end{align*}
For $p=3,$ according to \cref{e:p=n-k}
\[
    \sum_{1=r_0<r_1<r_2<r_3=4} [q^{r_0}, q^{r_1}, q^{r_2}, q^{r_3}]^s \prod_{i=0}^2 \coeff{r_{i+1}}{r_i}_\phi
    = [q, q^2, q^3, q^4]^s a_2 (3a_1 a_2) (6a_1^2 a_2) = 18\qbinom{s}{3} q^s a_2^3.
\]
Summing all terms gives
\begin{equation}\label{e:n=4}
    \begin{aligned}
        \coeff{4}{1}_{\phi^s}
        &= q^{s-1} [s]_{q^3} a_4 \\
        &{\qquad}+ 3 q^{s-2} \frac{(q^s-1)(q^{s-1}-1)(q^{s+1}+5q^s - 5 q^2 - 1)}{(q-1)(q^2-1)(q^3-1)} a_2^3 \\
        &{\qquad}+ 2 q^{s-1} \frac{(q^s-1)(q^{s-1}-1)(3q^{s+1} + 5 q^s + 5q + 2)}{(q^2 - 1)(q^3 - 1)} a_2 a_3.
    \end{aligned}
\end{equation}

\Cref{e:n=1,e:n=2,e:n=3,e:n=4} match the computations of the first few terms done by Korkine in 1882 \cite[p.~240]{korkine1882}. He used the ordinary coefficient convention, expressing $\inv{n!}\coeff{n}{1}_{\phi^s}$ in terms of the ordinary coefficients $q_n = a_n/n!$ up to $n=5.$

\section{The iterative logarithm}\label{s:itlog}

The iterative logarithm is a natural object of study in iteration theory \cite{ecalle1970,ecalle1974,jabotinsky1963}. It is defined by the equation
\begin{equation}
    \itlog(f) := \odv{f^s}{s}_{s=0},
\end{equation}
and is an analogue of the natural logarithm, with composition playing the role of multiplication. For instance, $\itlog(f^s) = s \itlog(f).$ This definition goes back to Frege (1874) \cite{frege1874,gronau1997}, but the iterative logarithm had already appeared, under a different form, in Collins's 1836 work \cite{collins1836}.

The iterative logarithm also arises naturally in umbral calculus. Indeed, the logarithm of the umbral operator $\phi$ of \Cref{s:uc} is given explicitly by $\x \itlog(f)(D),$ where $\x$ denotes multiplication by $x$ and $D$ differentiation with respect to $x$ \cite{beauduin2025}.

The convergence of the expansion of $\itlog(f)$ was shown to be equivalent to the embeddability of $f$ \cite{erdos1960}. The first terms of its (typically divergent) expansion were derived by Korkine \cite[p.~241]{korkine1882}, and the case $f'(0) = q=1$ was later fully determined by Jabotinsky \cite{jabotinsky1963}. We now present the expansion in wider generality, in forms similar to Schröder's formula (\Cref{t:gen}) and Tambs Lyche's formula (\Cref{t:tambs}).

\begin{theorem}\label[theorem]{t:logschroder}
     Suppose that $q=f'(0)$ is not a root of unity. The iterative logarithm of $f$ is given by
    \begin{equation}
        \itlog(f)(t) = t \log q + \frac{\log q}{q-1} \sum_{n=2}^\infty \frac{t^n}{n!} \sum_{p=1}^{n-1} \frac{(-1)^{p-1}}{[p]_q}  q^{-\binom{p+1}{2}}\coeff{n}{1}_{(\phi-q)_q^p},
    \end{equation}
    where the coefficient term is given by \cref{e:coeff1}.
\end{theorem}

\begin{theorem}\label[theorem]{t:logtambs}
     Under the same conditions, we have
    \begin{equation}
        \itlog(f)(t) = t \log q + \frac{\log q}{q-1} \sum_{n=2}^\infty \frac{t^n}{n!} \sum_{p=1}^{n-1} \frac{(-1)^{p-1}}{[p]_q} \qbinom{n-1}{p} q^{\binom{p+1}{2}-p n} \coeff{n}{1}_{\phi^p},
    \end{equation}
    where the coefficient term is given by \Cref{t:matrix}.
\end{theorem}

It is worth noting that the index of the first nonzero coefficient (the \emph{order}) is $1$ whenever $q\neq 1,$ whereas it is at least $2$ when $q = 1.$

\begin{proof}[Proof of \Cref{t:logschroder,t:logtambs}]
    To differentiate the fractional iterates, we first need an expression for the derivative of the $q$-binomial coefficient. For this purpose, we apply logarithmic differentiation:
    \[
    \odv*{\log \qbinom{s}{p}}{s}  = \sum_{\ell=0}^{p-1} \odv*{\log [s-\ell]_q}{s} = \sum_{\ell=0}^{p-1} \frac{q^{s-\ell}}{[s-\ell]_q} \frac{\log q}{q-1},
    \]
    and deduce that
    \begin{align*}
        \odv*{\qbinom{s}{p}}{s}
        &= \qbinom{s}{p} \frac{\log q}{q-1}\sum_{\ell=0}^{p-1} \frac{q^{s-\ell}}{[s-\ell]_q} \\
        &= \frac{\log q}{q-1}\pa{\frac{[s-1]_q \cdots [s-p+1]_q}{[p]_q!} q^s + \qbinom{s}{p} \sum_{\ell=1}^{p-1} \frac{q^{s-\ell}}{[s-\ell]_q}}.
    \end{align*}
    Evaluating the latter at $s=0$ yields
    \begin{align*}
        \odv*{\qbinom{s}{p}}{s}_{s=0}  
        &= \frac{\log q}{q-1} \frac{[-1]_q \cdots [-p+1]_q}{[p]_q!} \\
        &= \frac{\log q}{q-1} \frac{[1]_q \cdots [p-1]_q}{[p]_q!} \prod_{j=1}^{p-1} (-q^{-j}) = \frac{\log q}{q-1} \frac{(-1)^{p-1}}{[p]_q} q^{-\binom{p}{2}}.
    \end{align*}
    By the Leibniz product rule,
    \begin{equation}\label{e:Leibniz}
         \odv*{\qbinom{s}{p} q^s}{s} = \pa{\odv*{\qbinom{s}{p}}{s}} q^s + \qbinom{s}{p} q^s \log q.
    \end{equation}
    At $s=0,$ when $p > 0,$ the right-hand term of the addition in \cref{e:Leibniz} becomes zero, while for $p = 0,$ it is the left-hand term that vanishes. Therefore, according to \Cref{t:gen}
    \begin{align*}
        \odv*{\coeff{n}{1}_{\phi^s}}{s}_{s=0}
        &= (\log q) \coeff{n}{1}_1 + \sum_{p=1}^{n-1} \frac{\log q}{q-1} \frac{(-1)^{p-1}}{[p]_q} q^{-\binom{p}{2}} q^{-p} \coeff{n}{1}_{(\phi-q)_q^p} \\
        &= (\log q) \delta_{n-1} + \frac{\log q}{q-1} \sum_{p=1}^{n-1}\frac{(-1)^{p-1}}{[p]_q} q^{-\binom{p+1}{2}} \coeff{n}{1}_{(\phi-q)_q^p},
    \end{align*}
    hence the first formula holds. The second formula is derived in a manner analogous to the proof of \Cref{t:tambs}:
    \begin{align*}
        \sum_{p=1}^{n-1}\frac{(-1)^{p-1}}{[p]_q} q^{-\binom{p+1}{2}} \coeff{n}{1}_{(\phi-q)_q^p}
        &= \sum_{p=1}^{n-1}\frac{(-1)^{p-1}}{[p]_q} q^{-\binom{p+1}{2}} \sum_{\ell=0}^p \qbinom{p}{\ell} q^{\binom{p-\ell}{2}} (-q)^{p-\ell} \coeff{n}{1}_{\phi^\ell} \\
        &= \sum_{p=1}^{n-1} \sum_{\ell=1}^p \frac{(-1)^{\ell-1}}{[p]_q} \qbinom{p}{\ell} q^{\binom{p-\ell+1}{2} - \binom{p+1}{2}} \coeff{n}{1}_{\phi^\ell},
    \end{align*}
    where $\coeff{n}{1}_{\phi^0} = \delta_{n-1} = 0$ for $n \ge 2.$ Moreover, the $q$-binomial coefficient vanishes for $p < \ell,$ and using an absorption identity we arrive at
    \begin{align*}
        \sum_{p=1}^{n-1}\frac{(-1)^{p-1}}{[p]_q} q^{-\binom{p+1}{2}} \coeff{n}{1}_{(\phi-q)_q^p} &= \sum_{1\leq\ell\leq p<n} \frac{(-1)^{\ell-1}}{[\ell]_q} \qbinom{p-1}{\ell-1} q^{-\sum_{j=p-\ell+1}^p j} \coeff{n}{1}_{\phi^\ell} \\
        &= \sum_{\ell=1}^{n-1} \frac{(-1)^{\ell-1}}{[\ell]_q} \coeff{n}{1}_{\phi^\ell} \sum_{p=\ell}^{n-1} \qbinom{p-1}{\ell-1} q^{\binom{\ell}{2}-\ell p}.
    \end{align*}
    For the last step, we apply a $q$-\emph{hockey-stick identity}
    \begin{equation}
        \sum_{p=\ell}^{n-1} \qbinom{p-1}{\ell-1} q^{-\ell p} = \qbinom{n-1}{\ell} q^{-\ell (n-1)},
    \end{equation}
    which can be found in an equivalent form in \cite[Id.~3.1 with $p=1$]{shattuck2011}.
\end{proof}

In the limit $q \to 1,$ we recover the classic formulas \cite{jabotinsky1963}

\begin{align}
    \itlog(f)(t)
    &= \sum_{n=2}^\infty \frac{t^n}{n!} \sum_{p=1}^{n-1} \frac{(-1)^{p-1}}{p} \coeff{n}{1}_{(\phi-1)^p} 
    = \sum_{n=2}^\infty \frac{t^n}{n!} \sum_{p=1}^{n-1} \frac{(-1)^{p-1}}{p} \binom{n-1}{p} \coeff{n}{1}_{\phi^p}. \label{e:itlog1}
\end{align}

\begin{example}
    The function $f(t) = e^t-1$ was the first convergent power series shown to be non-embeddable, a result due to Baker \cite{baker1958}. The umbral operator $T$ associated with $f(t)$ maps the monomials to the \emph{Touchard polynomials} [\citenum{touchard1956},\,\citenum{beauduin2024}] whose coefficients are the \emph{Stirling numbers of the second kind} $\Stir{n}{k} = \coeff{n}{k}_T.$

    The coefficients $K_n,$ called \emph{Koszul numbers} (a name proposed in \cite{manetti2016}), are defined by the (divergent) expansion of the iterative logarithm of $f$:
    \[
    \itlog(f)(t) = \sum_{n=2}^\infty \frac{K_n}{n!} t^n.
    \]
    A formula for $K_n$ is given by \cref{e:itlog1}. Using \cref{e:coeff0} and the identity $\Stir{n}{1} = 1$ for $n > 0,$ it simplifies to
    \[
    K_n = \sum_{p=1}^{n-1} \frac{(-1)^{p-1}}{p} \sum_{1 < r_1 < \ldots < r_p = n} \prod_{i=1}^{p-1} \Stir{r_{i+1}}{r_i}.
    \]
    This formula was first shown by Aschenbrenner \cite{aschenbrenner2012}, in the process of proving a conjecture.
\end{example}

Perhaps the most interesting expansion is that in the form of Monkam's formula (\Cref{t:monkamgen}).
\begin{theorem}\label[theorem]{t:monkamlog}
    Suppose that $q=f'(0)$ is not a root of unity. The iterative logarithm of $f$ is given by
    \begin{equation}\label{e:monkamlog}
        \itlog(f)(t) = \sum_{n=1}^\infty \frac{t^n}{n!} \sum_{p=0}^{n-1}  \sum_{1=r_0<\ldots<r_p=n} \log[q^{r_0}, \dots, q^{r_p}] \prod_{i=0}^{p-1} \coeff{r_{i+1}}{r_i}_\phi.
    \end{equation}
\end{theorem}

\begin{proof}
    According to \cref{e:fdformula},
    \[
    \odv*{[x_0, \dots, x_p]^s}{s}_{s=0} = \sum_{i=0}^p \frac{\log(x_i)}{\prod_{j\neq i} (x_i - x_j)} = \log[x_0, \dots, x_p],
    \]
    which combined with \Cref{t:monkamgen} yields the desired result.
\end{proof}

More generally, \Cref{t:monkamlog,t:monkamgen} suggest defining an iterative analogue for a function $g$ by
\begin{equation}
    \text{it}g(f)(t) = \sum_{n=1}^\infty \frac{t^n}{n!} \sum_{p=0}^{n-1} \sum_{1=r_0<\ldots<r_p=n} g[q^{r_0}, \dots, q^{r_p}] \prod_{i=0}^{p-1} \coeff{r_{i+1}}{r_i}_\phi.
\end{equation}

{\small\bibliography{cimart}}

\end{document}